\documentclass[reqno,12pt]{amsart}
\usepackage{color} % black,white,red,green,blue,cyan,magenta,yellow
\usepackage{ifpdf}
\ifpdf
    \usepackage[pdftex]{graphicx}
    \usepackage[pdftex]{hyperref}
    \hypersetup{
        unicode=false,          % non-Latin characters in Acrobat’s bookmarks
        pdftoolbar=true,        % show Acrobat’s toolbar?
        pdfmenubar=true,        % show Acrobat’s menu?
        pdffitwindow=false,     % window fit to page when opened
        pdfstartview={FitH},    % fits the width of the page to the window
        pdftitle={SP Article},      % title
        pdfauthor={Sara Pollock},   % author
        pdfsubject={Mathematics},    % subject of the document
        pdfcreator={Sara Pollock},  % creator of the document
        pdfproducer={Sara Pollock}, % producer of the document
        pdfkeywords={PDE, analysis, numerical analysis}, % list of keywords
        pdfnewwindow=true,      % links in new window
        colorlinks=true,        % false: boxed links; true: colored links
        linkcolor=red,          % color of internal links
        citecolor=blue,         % color of links to bibliography
        filecolor=magenta,      % color of file links
        urlcolor=cyan           % color of external links
    }

\else
    \usepackage{graphicx}
    \usepackage{pstricks}
    
    \newcommand{\href}[2]{#2}
\fi

\usepackage{times}
\usepackage{amsfonts}
\usepackage[font={small}]{caption}

\usepackage{amsmath}
\usepackage{amsthm}
\usepackage{amssymb}
\usepackage{amsbsy}
\usepackage{amscd}
\usepackage{extarrows}
\usepackage{esint}

\usepackage{enumerate}
\newtheorem{theorem}{Theorem}[section]
\newtheorem{corollary}[theorem]{Corollary}
\newtheorem{lemma}[theorem]{Lemma}

\newtheorem{assumption}[theorem]{Assumption}
\newtheorem{proposition}[theorem]{Proposition}

\newtheorem{definition}[theorem]{Definition}

\newtheorem{remark}[theorem]{Remark}

\numberwithin{equation}{section}  %amsmath command: tie counter to section

\definecolor{myblue}{rgb}{0.2,0.2,0.7}
\definecolor{mygreen}{rgb}{0,0.6,0}
\definecolor{mycyan}{rgb}{0,0.6,0.6}
\definecolor{myred}{rgb}{0.9,0.2,0.2}
\definecolor{mymagenta}{rgb}{0.9,0.2,0.9}
\definecolor{mywhite}{rgb}{1.0,1.0,1.0}
\definecolor{myblack}{rgb}{0.0,0.0,0.0}

\renewcommand{\div}{{\operatorname{div}}}
\renewcommand{\d}{{\operatorname{d}}}
\newcommand{\eps}{\varepsilon}
\newcommand{\R}{{\mathbb R}}       % Real numbers.
\newcommand{\cJ}{{\mathcal J}}
\newcommand{\cK}{{\mathcal K}}
\newcommand{\cQ}{{\mathcal Q}}
\newcommand{\cT}{{\mathcal T}}
\newcommand{\cV}{{\mathcal V}}
\DeclareMathAlphabet{\mathpzc}{OT1}{pzc}{m}{it}

\newcommand{\f}{\frac}
\newcommand{\on}{\text{ on }}
\newcommand{\an}{\text{ and }}
\newcommand{\inn}{\text{ in }}
\newcommand{\tforall}{\text{ for all }}
 
\newcommand{\ale}{\text{a.e.\,}} 
\newcommand{\card}{\text{card}}
\newcommand{\rest}{\big|}

\newcommand{\grad}{\nabla} 
\newcommand{\tdot}{\!\cdot\!}
\newcommand{\goto}{\rightarrow}

\newcommand{\pa}{\partial}

\definecolor{blue}{rgb}{0.2,0.2,0.7}
\definecolor{red}{rgb}{0.7,0.3,0.1}
\definecolor{cyan}{rgb}{0.2,0.5,0.6}
\usepackage{mathtools}
\usepackage{stmaryrd}
\begin{document}

\title[Matrix analysis for comparison principles]
      {A matrix analysis approach to discrete comparison principles for 
       nonmonotone PDE}

\author[S. Pollock]{Sara Pollock}
\email{sara.pollock@wright.edu}

\author[Y. Zhu]{Yunrong Zhu}
\email{zhuyunr@isu.edu}

\address{Department of Mathematics and Statistics\\
         Wright State University\\ 
         Dayton, OH 45435}

\address{Department of Mathematics and Statistics\\
         Idaho State University\\ 
         Pocatello, ID 83209}

\thanks{SP was supported in part by NSF DMS 1719849. 
        YZ was supported in part by NSF DMS 1319110.}

\date{\today}


%AMS subject classifications (MC2010):
%65N12  Stability and convergence of numerical method
%65N22  Solution of discretized equations [See also 65Fxx, 65Hxx]
%65N30  Finite elements, Rayleigh-Ritz and Galerkin methods, finite methods
%35J15  Second-order elliptic equations
%35J62  Quasilinear elliptic equations
%35J92  Quasilinear elliptic equations with $p$-Laplacian
%35J93  Quasilinear elliptic equations with mean curvature operator

%AMS subject classifications: 65N30,35J62,

\begin{abstract}
We consider a linear algebra approach to establishing a discrete comparison principle 
for a nonmonotone class of quasilinear elliptic partial differential equations.  
In the absence of a lower order term, we require local conditions on the mesh 
to establish the comparison principle and uniqueness of the piecewise linear 
finite element solution.
We consider the assembled matrix corresponding to the linearized problem satisfied by 
the difference of two solutions to the nonlinear problem.
Monotonicity of the assembled matrix establishes a maximum principle for the linear
problem and a comparison principle for the nonlinear problem.
The matrix analysis approach to the discrete comparison principle yields sharper
constants and more relaxed mesh conditions than does the argument by contradiction
used in previous work.

\end{abstract}

\maketitle

%%%%%%%%%%%%%%%%%%%%%%%%%%%%%%%%%%%%%%%%%%%%%%%%%%%%%%%%%%%%%%%%%%%%%%%%%%%%%%
\section{Introduction}\label{sec:intro} 

We consider a linear algebra approach to develop a discrete comparison principle for 
the equation
\begin{align}\label{eqn:QLPDEstrong}
-\div(\kappa(x,u)\grad u) + g(x,u)= f, \inn \Omega \subset \R^2, 
\end{align}
with homogeneous Dirichlet conditions
\begin{align}\label{eqn:diriBC}
u = 0 \on \Gamma_D = \pa \Omega.
\end{align}
The discrete comparison principle 
directly implies uniqueness of the discrete solution, in agreement with
the comparison principle and uniqueness for the continuous problem.

The PDE \eqref{eqn:QLPDEstrong} is both nonmonotone and nonvariational 
(see, {\em e.g.,} \cite{HlKrMa94}); and, as demonstrated in \cite{AnCh96a}, 
uniqueness of solutions to its finite element approximation can fail if the mesh 
is too coarse, even where the PDE solution is known to be unique.
Asympototic error estimates for a finite element approximation as the meshsize 
$h \goto 0$ were first shown in 1975 in \cite{DoDu75}. 
More recently, similar results were shown to hold
under integration by quadrature in \cite{AbVi12}.
In \cite{AnCh96a}, an argument by contradiction related to the approach
used in the continuous case is used to establish a discrete comparison principle
based on the condition that the mesh partition is globally fine enough.
The current authors used similar ideas in \cite{PoZh17b,PoZh17a} to demonstrate
that a local verifiable condition based on the variance of the solution over each 
element, rather than a global meshize condition, is sufficient for uniqueness of 
solutions in the absence of a lower order term.  
Here, we improve the constant appearing in the
condition and also relax the angle condition on the mesh.

This manuscript is motivated by the linear algebra approach used to establish 
a discrete maximum principle in \cite{BrKoKr08}, demonstrated to improve previously 
established constants.  The current authors also showed in  \cite{PoZh17b} 
that the maximum principle 
for the linear reaction-diffusion equation developed in \cite[Theorem 3.8]{BrKoKr08} 
has a direct application to a discrete comparison principle for the semilinear problem, 
$-\Delta u + g(x,u) = f$; and, the matrix analysis approach yields an improved constant. 
Here, we extend the analysis to quasilinear problems.  
In the semilinear case, the argument 
follows by showing that the assembled matrix in question is a {\em Stieltjes} matrix,
which is a symmetric positive definite matrix with nonpositive off-diagonal entries.
The particulars of the analysis do not apply in the quasilinear case, as the 
corresponding coefficient matrix for the linearization of \eqref{eqn:QLPDEstrong} 
is easily seen to be nonsymmetric, hence not {\em Stieltjes}.  

In this paper, we develop conditions under which the assembled coefficient matrix $A$
corresponding to the PDE satisfied by the difference between a subsolution and 
supersolution of \eqref{eqn:QLPDEstrong} is monotone. 
We proceed by first showing it is a $Z$-matrix, one with 
nonpositive off-diagonal entries; then, showing there is a diagonal matrix $D$ such
that $AD$ is strictly diagonally dominant, satisfying a condition for monotonicity
found in \cite{FiPe62,Plemmons77}.
The first main contribution of this work is improving the constants 
in the local condition for the discrete comparison principle hence uniqueness to hold.
The second main contribution is establishing the discrete comparison principle 
holds for problem \eqref{eqn:QLPDEstrong} on meshes with at least some right angles. 

In previous work by the authors \cite{PoZh17b,PoZh17a}, the mesh was assumed acute,
meaning all interior angles were bounded below $\pi/2$.  In the current results,
interior angles can be no greater than $\pi/2$, and opposite angles across each edge 
must sum to less that $\pi$. 
It is well known (see, for instance \cite[Lemma 2.1]{Xu.J;Zikatanov.L1999}), 
that for the assembled matrix for the Laplacian, monotonicity holds
under the condition that the mesh is Delaunay, meaning the angles opposite each edge
sum to no more than $\pi$. More general geometric conditions for a discrete maximum
principle for Poisson's equation are developed in \cite{DrDuSc04}, in which 
certain refinements of meshes satisfying monotonicity conditions are shown to remain
monotone.  The stronger condition on the geometry in this work comes from 
the variable-dependence in the principal part of the linearized problem, so that
improved conditions for Laplace operators do not directly apply. 

The theory of monotone matrices in numerical analysis has been well-studied, largely
with respect to the convergence of iterative methods \cite{Varga.R2000,Young71}.
In \cite{Plemmons77}, a collection of 40 conditions for a $Z$-matrix to be a
nonsingular $M$-matrix,
hence monotone, is drawn from the literature.  We use one of those conditions,
which appears earlier in \cite{FiPe62}, to establish our results.  The main 
technical challenge of this work is to understand the monotonicity of the
coefficient matrix arising from the discretization of a linear 
convection-diffusion, or reaction-convection-diffusion equation. 
This matrix is nonsymmetric, and lacks strict diagonal dominance.
The failure in the linear case for a convection-diffusion
coefficient matrix to have strictly positive row-sums is also discussed in 
\cite{Xu.J;Zikatanov.L1999}.  

The comparison principle for \eqref{eqn:QLPDEstrong} is equivalent to a maximum
principle for a linearized equation satisfied by the difference of a subsolution and 
supersolution to \eqref{eqn:QLPDEstrong}, which is demonstrated here by the
monotonicity of the assembled coefficient matrix. This relationship and its
implication for the uniqueness of the finite element solution is summarized
in Theorem \ref{thm:comparison}. 
 
The remainder of the paper is structured as follows.  In Section \ref{sec:prelim},
the discretization and discrete comparison principle are introduced.
Then, the linearized problem used to investigate the comparison principle is derived. 
Theorem \ref{thm:comparison} summarizes the relationship between the monotonicity 
of the assembled matrix for the linearized problem and the comparison principle 
for the nonlinear problem. The definitions of $Z$, $L$ and $M$ matrices, 
and the relevant theorems on their relationships are recalled from the literature. 
Section \ref{sec:Mcond} contains the technical estimates used to show the assembled 
matrix $A$ is monotone.

\section{Preliminaries}\label{sec:prelim}
We make the following assumptions on the problem data $\kappa(\cdot \, , \cdot)$  and
$g(\cdot\, , \cdot)$.
\begin{assumption}\label{A1:elliptic}
Assume $\kappa(x,\eta)$ and $g(x,\eta)$ are Carath\'eodory functions, measurable in $x$
for each $\eta \in \R$, and $C^1$ in $\eta$ for $\ale x \in \Omega$. Assume 
there are constants $0 < k_\alpha < k_\beta$ with
\begin{align}\label{eqn:A1:elliptic}
k_\alpha \le \kappa(x, \eta) \le k_\beta,
\end{align} 
for all $\eta \in \R$, and $\ale x \in \Omega$. Assume there is a positive $K_\eta$ with
\begin{align}\label{eqn:A1:Keta}
\left| \f{\pa \kappa}{\pa \eta}(x,\eta) \right| \le K_\eta, 
\end{align}
for all $\eta \in \R$ and $\ale x \in \Omega$.
Assume $g(x,\eta)$ is nondecreasing with respect to its second argument, and
there is a constant $G_\eta$ with
\begin{align}\label{eqn:A1:Geta}
0 \le \f{\pa g}{\pa \eta}(x,\eta) \le G_\eta,
\end{align}
for all $\eta \in \R$ and $\ale x \in \Omega$.
\end{assumption}
Under Assumption \ref{A1:elliptic}, the
PDE is known to satisfy a comparison principle and  have a unique solution, 
as demonstrated in \cite{DoDuSe71,Trudinger74} and \cite[Chapter 10]{GilbargTrudinger}.
Additionally, the weak form of \eqref{eqn:QLPDEstrong} is 
given as follows for $V = H_{0}^1(\Omega)$, the closed subspace of $H^1$ with vanishing
trace on $\pa \Omega$. Find $u \in V$ such that
\begin{align}\label{eqn:QLPDEweak}
\int_\Omega \kappa(x,u) \grad u \tdot \grad v + g(x,u)v \, \d x = 
\int_\Omega fv \, \d x, \tforall v \in V.
\end{align}
The functional setting of the weak form can be understood in the context of the
Leray-Lions conditions for pseudomonotonicity.  We refer interested readers to
\cite[Chapter 2-3]{CarlLeMontreanu} for further details. We note that the class of
problems defined by the assumptions in this section is called {\em nonmonotone} because
the inequality 
\[\int_\Omega (\kappa(x,w)\grad w -\kappa(x,v) \grad v)\tdot \grad(w-v) \, \d x\ge 0,\]
does not in general hold for all $w, v \in V$, 
even for $\kappa = \kappa(u)$, with $\kappa'(u) \ge 0$.

\subsection{Discretization}\label{subsec:discret}
Let $\cT$ be a conforming simplicial partition of domain $\Omega$ that exactly 
captures the boundary.
Let $\overline \cQ$ be the collection of vertices or nodes of $\cT$, and let
$\cQ =\overline \cQ \setminus \pa \Omega$ be the set nodes that do not lie on the
Dirichlet boundary, corresponding to the mesh degrees of freedom.
Let $\cV \subsetneq V$ be the discrete space spanned by
the piecewise linear basis functions $\{\varphi_j\}$ that satisfy
$\varphi_i(q_i) = 1$ and 
$\varphi_i(q_j) = 0$ for each $q_j \in \overline \cQ$ with $q_j \ne q_i$.
Define the non-negative subset of $\cV$ by 
$\cV^+ \coloneqq \{ v \in \cV ~\rest~ v \ge 0\}$.

Let $\omega_i$ be the support of the basis function $\varphi_i$. Define the intersection
of support for any two basis functions with respect to a global numbering by
$\omega_{ij} = \omega_i \cap \omega_j$. In terms of the corresponding nodes $q_i$
and $q_j$, it follows that $\omega_{ij}$ is the union of elements that share both
$q_i$ and $q_j$ as vertices.
\begin{align*}
\omega_{ij} = \bigcup \{T \in \cT ~\rest~ q_i \in T ~\an q_j \in T\}.
\end{align*}

Additional notation for the discretization is summarized as follows, and illustrated
in Figure \ref{fig:diags}.
\begin{itemize}
\item $e_{ij}$, denotes the edge connecting vertices $q_i$ and $q_j$.
\item $\theta_{ij}^+$ and $\theta_{ij}^-$ denote the two respective angles opposite
      edge $e_{ij}$ in $\omega_{ij}$.
\item $\theta_{j,T}$ denotes the interior angle of triangle $T$ at vertex $q_j$, and
      $\theta_{ij,T}$ denotes the interior angle opposite vertices $q_i$ and $q_j$ in
      triangle $T$.
\item $q_{ij}^T$ denotes the vertex opposite both $q_i$ and $q_j$ in triangle $T$.
\item $\varphi_{ij}^T$ denotes the basis function associated with $q_{ij}^T$ in
      triangle $T$.
\item $\delta^e_{ij}(v) = |v(q_i) - v(q_j)|$ is the absolute difference of nodal
      values over edge $e_{ij}$.
\item $\delta_{\omega}^{(i)}(v) = \max\{\delta_{jk}^e(v) ~\rest~ q_j, q_k \in \omega,
      ~j,k \ne i  \}$ 
      is the maximum
      difference between neighboring nodal values of $v$ in $\omega$, 
      over each edge not touching vertex $q_i$.
\item $\delta_{\omega}(v) = \max\{\delta_{jk}^e(v) ~\rest~  q_j, q_k \in \omega\}$ 
      denotes the maximum difference between neighboring nodal values of $v$ in $\omega$.
\item $|T|$ denotes the area of triangle $T \in \cT$.
\item $|T_\cT| = \max_{T \in \cT}|T|$.
\item $\overline\cQ_i= \{q_j \in \overline \cQ ~\rest~ q_j \in \omega_{i}, ~j \ne i\}$, 
denotes the set of vertices neighboring $q_i$, including those on $\Gamma_D$.
\item $\cQ_i= \overline \cQ_i \setminus \Gamma_D$, the set of non-Dirichlet vertices 
neighboring $q_i$.
\end{itemize}

\begin{figure}
\centering
\includegraphics[trim=5pt 5pt 5pt 5pt, clip=true,width=0.8\textwidth]
{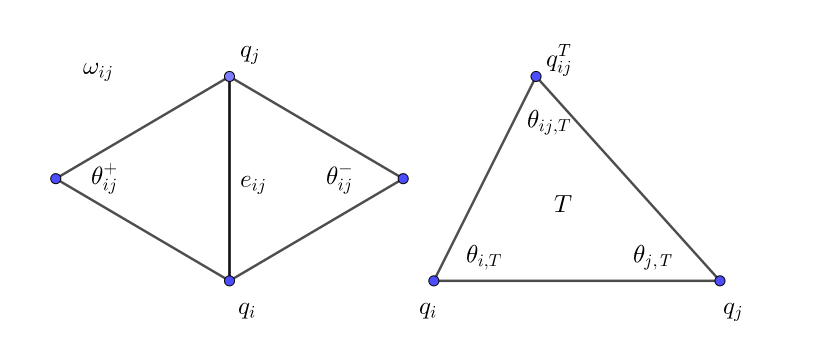}
\caption{Left: Schematic of a patch. Right: Schematic of an element.}
\label{fig:diags}
\end{figure}

\begin{assumption}\label{A2:mesh}
It is assumed any interior angle of the mesh satisfies $\theta \le \pi/2$, and 
that any two angles opposite an edge must sum to less than $\pi$. 
In particular, there is a constant $\beta_m > 0$ for which
\begin{align}\label{eqn:opp-angle-cond}
\cot \theta_{ij}^+ + \cot \theta_{ij}^- > \beta_m, ~\text{ for each }~ 
\omega_{ij} \subset \cT.
\end{align}
It is also assumed that the mesh satisfies a smallest-angle condition over each
neighborhood $\omega_{ij}$. There is a constant $\beta_M > 0$ for which
\begin{align}\label{eqn:full-angle-cond}
\cot \theta_{ij}^+ + \cot \theta_{ij}^- + 
\sum_{T \in \omega_{ij}} \cot \theta_{i,T} \le \beta_M
~\text{ for each }~ \omega_{ij} \subset \cT.
\end{align}
\end{assumption}

\subsection{Comparison framework}\label{subsec:framework}
Consider the problems: Find $u_i \subset \cV$ such that
\begin{align}\label{eqn:QLPDEdiscreteF}
\int_\Omega\kappa(x,u_i) \grad u_i \tdot \grad v +g(x,u_i)v \, \d x 
= \int_\Omega f_i v \, \d x, 
\tforall v \in \cV, 
\end{align}
for $f_i \in L_2(\Omega)$, $i = 1,2$.
The discrete comparison principle for \eqref{eqn:QLPDEdiscreteF} states that 
whenever $f_1 \le f_2$, ($\ale x \in \Omega$), meaning
$\int_\Omega (f_1 - f_2)v \, \d x \le 0$,
for each $v \in \cV^+$, then it holds that $u_1 \le u_2$.

The comparison principle can be restated in terms of a 
maximum principle for $w = u_1 - u_2$. The discrete problem satisfied by $w$ can
be understood by applying Taylor's theorem to the difference of 
\eqref{eqn:QLPDEdiscreteF} with $i = 1$ and $i = 2$.
\begin{align}\label{eqn:a1_001}
&\int_\Omega \left( \kappa(x,u_1) \grad u_1 \tdot \grad v + g(x,u_1)v \right) \, \d x- 
\int_\Omega \left( \kappa(x,u_2) \grad u_2 \tdot \grad v + g(x,u_2)v \right)\, \d x 
\nonumber \\
& = \int_\Omega \kappa(x,u_1) \grad (u_1 - u_2) \tdot \grad v 
    - (\kappa(x,u_2) - \kappa(x,u_1)) \grad u_2 \tdot \grad v \, \d x
\nonumber \\
&  + \int_\Omega (g(x,u_1) - g(x,u_2)) v \, \d x
\nonumber \\
& = \int_\Omega \kappa(x,u_1)\grad (u_1 - u_2) \tdot \grad v \, \d x 
    - \int_\Omega \left( \int_0^1 
     \f{\pa \kappa}{\pa \eta}(x,z(t)) (u_1 - u_2) \, d t\right) 
      \grad u_2 \tdot  \grad v \, \d x
\nonumber \\ 
&  + \int_\Omega\left( \int_0^1 
     \f{\pa g}{\pa \eta}(x,z(t)) (u_1 - u_2)  \, dt \right)  v \, \d x, 
\end{align}
where $z(t) = t u_1 + (1-t)u_2$, for $0 \le t \le 1$.
Letting $w = u_1 - u_2$, the equation satisfied by $w$ is then
\begin{align}\label{eqn:a1_002}
&\int_\Omega \kappa(x,u_1)\grad w \tdot \grad v \, \d x 
    +\int_\Omega\left( \int_0^1  \f{\pa \kappa}{\pa \eta}(x,z(t)) \, d t\right) w 
      \grad u_2 \tdot \grad v \, \d x
\nonumber \\ 
&  + \int_\Omega\left(\int_0^1 \f{\pa g}{\pa \eta}(x,z(t)) w  \, dt \right) v \, \d x 
= \int_\Omega (f_1 - f_2) v \, \d x, ~\tforall v \in \cV. 
\end{align}

The comparison principle for $u_1$ and $u_2$ in \eqref{eqn:QLPDEdiscreteF} is then
equivalent to the weak maximum principle for $w = u_1 - u_2$ in 
\eqref{eqn:a1_002}, namely $w \le 0$ whenenever $f_1 - f_2 \le 0$.
We now turn our attention to establishing the weak maximum principle for $w$. 

\section{Discrete maximum principle}
From \eqref{eqn:a1_002}, the linear equation for $w$ is a general second order 
elliptic equation with convection and reaction terms
\begin{align}\label{eqn:rcd001}	
\int_{\Omega} \kappa(x, u_{1}) \nabla w\cdot \nabla v \,  \d x 
+ \int_{\Omega} \vec{b}(x, u_{2}) w\tdot\nabla v \, \d x + \int_{\Omega} c(x) wv \, \d x 
= \int_{\Omega}(f_{1}-f_{2}) v \, \d x,
\end{align}
for all $v \in \cV$, with
\begin{align}\label{eqn:rcd002}
\vec b(x) & \coloneqq b(x) \grad u_2(x), ~\text{ with }~ 
b(x) \coloneqq\int_0^1 \f{\pa \kappa}{\pa \eta}(x,z(t)) \, \d t,
\\ \label{eqn:rcd003}
     c(x) & \coloneqq \int_0^1 \f{\pa g}{\pa \eta}(x, z(t)) \, \d t.
\end{align}

We now turn our attention to the properties of the assembled system
\eqref{eqn:rcd001}-\eqref{eqn:rcd003}.
The discrete function $w \in \cV$ has the expansion in basis functions
$w = \sum_{j=1}^{n} W_{j}\varphi_{j}$, where $n$ is the number of mesh degrees of
freedom. Choosing the test functions $v = \varphi_{i}$ for each $i=1, \cdots, n$ 
in Equation \eqref{eqn:rcd001}, we obtain the equivalent matrix problem $AW = F$.
In particular, $W = (W_{1}, \cdots, W_{n})^{T}$, $F=(F_{1}, \cdots, F_{n})^{T}$ with 
$A =(a_{ij})$ and $F_i$ defined entrywise by
\begin{align}\label{eqn:assemA}
a_{ij} &= \int_{\omega_{ij}} \kappa(x, u_{1})
         \nabla \varphi_j \tdot  \nabla \varphi_i \, \d x +
\int_{\omega_{ij}} \vec{b}(x, u_{2}) \varphi_j \tdot\nabla \varphi_i \, \d x
+ \int_{\omega_{ij}} c(x) \varphi_j \varphi_i \, \d x
\\ \label{eqn:assemF}
F_{i} &= \int_{\omega_{i}} (f_{1}-f_{2}) \varphi_{i} \,  \d x.
\end{align}

We now investigate the maximum principle for $w$ through the monotonicity 
of matrix $A$ given by \eqref{eqn:assemA}.  
\begin{definition}[Monotone matrix]\label{def:monoMat}
A real square matrix $A$ is monotone 
(in the sense of \cite[Section 23]{Collatz.L1966} ) if for all real vectors 
$v$, $Av\ge 0$ implies $ v\ge 0$, 
where $\ge$ is the element-wise ordering.
\end{definition}
By this definition, a monotone matrix is nonsingular because
if $A$ is a monotone matrix and $x$ is a vector in its nullspace, then both
$x\ge 0$ and $-x\ge 0$, implying $x = 0$.
We mention another relevant property of monotone matrices. 
\begin{proposition}
\label{prop:monotone}
$A$ is monotone iff $A$ is invertible and $A^{-1} \ge 0$.
\end{proposition}
\begin{proof}
If $A$ is monotone, then $A$ is nonsingular and $A^{-1}$ exists. 
Let $x$ be the $i$th column of $A^{-1}$, so we have $Ax = e_{i}$, 
the $i$-th standard basis vector. 
This shows $Ax \ge 0$ which implies $x\ge 0.$ Therefore $A^{-1} \ge 0.$

Reversely, suppose $A$ has an inverse and $A^{-1}\ge 0.$ If $Ax\ge 0$, 
then $x = A^{-1}Ax \ge A^{-1} 0 = 0.$ Hence $A$ is monotone. 
\end{proof}

The next theorem summarizes how monotonicity of the assembled
matrix for $w$ implies the comparison principle, hence uniqueness of the solution.
\begin{theorem}\label{thm:comparison}
Suppose the functions $\kappa(x,\eta)$ and $g(x,\eta)$ are each measurable with 
respect to the first argument, and $C^1$ with respect to the second.
Let $u_i \subset \cV$ solve
\begin{align}\label{eqn:QLPDEdiscrete}
\int_\Omega\kappa(x,u_i) \grad u_i \tdot \grad v +g(x,u_i)v \, \d x 
= \int_\Omega f_i v \, \d x, 
\tforall v \in \cV, 
\end{align}
for $f_i \in L_2(\Omega)$, $i = 1,2$, with $\int_\Omega (f_1 - f_2)v \, \d x \le 0$,
for all $v \in \cV^+$.
Let $A$ be the coefficient matrix defined by \eqref{eqn:assemA}, and
$F$ the vector defined by \eqref{eqn:assemF}.
Then if $A$ is monotone, $u_1 \le u_2$. Moreover, if $f_1 = f_2$, then $u_1 = u_2$.
\end{theorem} 
\begin{proof}
Let the discrete function $w \in \cV$ be given by $w = \sum_{j = 1}^n W_j \varphi_j$, 
with $W = (W_{1}, \cdots, W_{n})^{T}$, where $n = \card(\cQ)$, 
the number of mesh degrees of freedom. The monotonicity of $A$ implies its invertibility.
By equations \eqref{eqn:a1_001}-\eqref{eqn:a1_002}, and the definitions of 
$A$ and $F$, the vector $W$ that solves $AW = F$ uniquely defines $w \in \cV$ that
satisfies $w = u_1 - u_2$.

Since $A$ is monotone, $F \le 0$ implies  $W \le 0$ which by the
nonnegativity of basis functions implies $w \le 0$ implying $u_1 \le u_2$.
If $f_1 = f_2$, it follows that $u_1 = u_2$.
\end{proof}

\begin{remark} The monotonicity of matrix $A$ given by \eqref{eqn:assemA} also, by
the reasoning above, establishes a discrete maximum principle for the linear 
problem \eqref{eqn:rcd001}.
\end{remark}

In the remainder of the paper we develop conditions under which the assembled matrix
\eqref{eqn:assemA} is monotone.
Three related classes of matrices we enounter in the proof are $Z$-matrices,
$L$-matrices and $M$-matrices, defined as follows.

\begin{definition}[$Z$-matrix]\label{def:Zmatrix}
A matrix $A$ is called an $Z$-matrix if $a_{ij}\le 0$ when $i\neq j$ 
(see \cite[Definition 4,1]{FiPe62},\cite{Plemmons77}).
\end{definition}
A $Z$-matrix with positive diagonal entries is also called an $L$-matrix.  
\begin{definition}[$L$-matrix]\label{def:Lmatrix}
A matrix $A$ is called an $L$-matrix if $a_{ii}>0$ for all $i$ ,
and $a_{ij}\le 0$ when $i\neq j$ (see \cite[Definition 2.9]{Windisch.G1989},
\cite[Chapter 2, Definition 7.1]{Young71})
\end{definition}
A monotone $Z$-matrix is an $M$-matrix.
\begin{definition}[{$M$-matrix, see \cite[Definition 3.22]{Varga.R2000}}]
\label{def:Mmatrix}
A real $n \times n$ matrix $A = (a_{ij})$ with $a_{ij} \le 0$ for all $i\neq j$ is an $M$-matrix if $A$ is nonsingular and $A^{-1} \ge 0$.
\end{definition}
We note the equivalence of this definition to \cite[Definition 7.3]{Young71}, and 
to a nonsingular $M$-matrix, as in \cite{Plemmons77}.
If the off-diagonal entries of $A$ are nonpositive, then $A$ is monotone if and only 
if $A$ is an $M$-matrix. 
This is clear from the Proposition~\ref{prop:monotone} of the monotone matrix, 
and the definition of the $M$-matrix. 

In what follows, we will show the monotonicity
of the coefficient matrix $A$ given by \eqref{eqn:assemA}, by first showing $A$ is a 
$Z$-matrix: the off-diagonal entries are nonpositive; then showing it is
monotone, by the following 1962 result of M. Fiedler and V. Pt\'ak.  
We first recall the standard definition of (strict) diagonal dominance.
\begin{definition}\label{def:diag-dom}
An $n \times n$ matrix is strictly diagonally dominant if
\begin{align}\label{eqn:diag-dom}
|a_{ii}| > \sum_{j=1, j \ne i}^n|a_{ij}|, ~\tforall~ i = 1, \ldots, n. 
\end{align}
Strict diagonal dominance may also be called strong diagonal dominance 
\cite[p.41]{Young71}.  A matrix is called diagonally dominant 
(or weakly diagonally dominant) if equality is allowed in \eqref{eqn:diag-dom}, 
with a strict inequality for at least one index $i$.
\end{definition}
The full theorem of \cite{FiPe62} presents 13 equivalent
statements, here we paraphrase the two most relevant to our purposes.
\begin{theorem}\cite[Theorem 4,3.4]{FiPe62},\cite[Theorem 1 (N39)]{Plemmons77}
\label{thm:FiPe}
Let $A$ be a $Z$-matrix. Then $A$ is monotone if and only if there exists a diagonal 
matrix $D$ with positive diagonal elements
such that the matrix $AD$ is strictly diagonally dominant.
\end{theorem}

Along the way to showing the monotonicity, we also show $A$ is an 
$L$-matrix: the diagonal entries are strictly positive.  As as consequence of the 
monotonicity, $A$ is then also an $M$-matrix, as in the corresponding result
of \cite{Plemmons77}.  
There are many characterizations of $Z$-matrices as $M$-matrices, as 
in the review paper \cite{Plemmons77}, the references therein, 
and \cite[Chapter 3]{Varga.R2000}.  Further discussion on the conditions relating
$Z$-matrices, $L$-matrices and $M$-matrices may be found in \cite[Chapter 2]{Young71},
and \cite[Chapter 2]{Windisch.G1989}. 
For completeness, we also mention the results of F. Bouchon \cite{Bouchon07},
which depend explicity on irreducibility properties of the assembled matrix.
Referring to the counterexample of \cite{DrDuSc04}, the irreducibility can fail to hold 
in certain situations, even on a connected mesh.
As the assembled matrix from \eqref{eqn:assemA} is not itself strictly diagonally 
dominant, and it is not obvious how to compute the spectral radius by direct means, 
we choose Theorem \ref{thm:FiPe} as the simplest condition to demonstrate based on
computationally verifiable conditions.

\section{Properties of the assembled matrix}\label{sec:Mcond}
In this section, we develop conditions under which $A$, the coefficient matrix from 
\eqref{eqn:assemA}, is a $Z$-matrix that satisfies Theorem \ref{thm:FiPe}, hence is
monotone. Our main comparison result then follows from Theorem \ref{thm:comparison}. 
Each entry $a_{ij}$ of matrix $A$ is given by \eqref{eqn:assemA}
with $\vec b$ and $c$ given by \eqref{eqn:rcd002} and \eqref{eqn:rcd003}, respectively.
We now consider the conditions under which matrix $A$ has non-positive off-diagonal 
entries, meaning $A$ is a $Z$-matrix.

\begin{lemma}\label{lem:pij-nonpos}
Let Assumption \ref{A1:elliptic} and Assumption \ref{A2:mesh} hold. 
Let $A = (a_{ij})$ be given by \eqref{eqn:assemA}.
On satisfaction of the condition on the nodal values of $u_2$ and meshsize $|T_\cT|$ 
\begin{align}\label{eqn:pij-cond}
\f{1}{3 \beta_m} \left\{ K_\eta \beta_M \delta_{\omega_{ij}}(u_2) 
  + G_\eta |T_\cT| \right\} 
  \le k_\alpha, ~\text{ for each }~ \omega_{ij} \subset \cT,
\end{align}
it holds that $a_{ij} \le 0$ for each $i \ne j$.
\end{lemma}
\begin{proof}
Equivalently, and for convenience of later calculations, we consider the entry $a_{ji}$.
By direct calculation, it holds for $i \ne j$ that. 
\begin{align}\label{eqn:pij002}
\grad \varphi_i \tdot \grad \varphi_j \rest_T =  \f {-1} {2|T|} \cot \theta_{ij,T},
\end{align}
where $\theta_{ij,T}$ is the angle opposite edge $e_{ij}$ in triangle $T$.
Equation \eqref{eqn:pij002} together with \eqref{eqn:A1:elliptic} and 
\eqref{eqn:opp-angle-cond} implies
\begin{align}\label{eqn:pij002a}
\int_{\omega_{ij}} \kappa(x,u_1) \grad \varphi_i \tdot \grad \varphi_j \, \d x
\le k_\alpha \sum_{T \in \omega_{ij}} \int_T\grad \varphi_i \tdot \grad \varphi_j \, \d x 
=  -\f {k_\alpha} {2} (\cot \theta_{ij}^+ + \cot \theta_{ij}^-),
\end{align}
for interior nodes, $q_i, q_j \in \cQ$.
For the reaction term, $\int_T \varphi_i \varphi_j \, \d x = {|T|}/{12}$.
Together with \eqref{eqn:A1:Geta} and \eqref{eqn:rcd003}, this implies
\begin{align}\label{eqn:pij003a}
\int_{\omega_{ij}} c(x) \varphi_i \varphi_j \, \d x
\le \sum_{T \in \omega_{ij}} \f{G_\eta|T|}{12},
\end{align}
for $q_i, q_j \in \cQ$.  

To bound the  nonsymmetric term, the following decomposition is useful.
Over each triangle $T$ with vertices $q_i, q_i, q_k$ and discrete function $v \in \cV$, 
we have
\begin{align}\label{eqn:pij004}
\grad v = (v(q_i) - v(q_j)) \grad \varphi_i 
          + (v(q_k) - v(q_j))\grad \varphi_k.
\end{align}
Applying \eqref{eqn:pij002} and \eqref{eqn:pij004} with $v = u_2$ and 
$q_k = q_{ij}^T$ for each $T \in \omega_{ij}$ yields 
\begin{align}\label{eqn:pij005}
\int_{\omega_{ij}} b(x)\grad u_2\tdot \grad \varphi_j \varphi_i \, \d x
& =  (u_2(q_i) - u_2(q_j)) \int_{\omega_{ij}}\grad \varphi_i \tdot \grad \varphi_j\, 
     b(x) \varphi_i \, \d x 
\nonumber \\
& + \sum_{T \in \omega_{ij}} (u_2(q_{ij}^T) - u_2(q_j))
   \int_T\grad \varphi_{ij}^T \tdot \grad \varphi_j \, b(x)\varphi_i \, \d x 
\nonumber \\
& \le \f {K_\eta} 6 |u_2(q_i) - u_2(q_j)|(\cot \theta_{ij}^+ + \cot \theta_{ij}^-) 
\nonumber \\
& + \f {K_\eta} 6 \sum_{T \in \omega_{ij}} |u_2(q_{ij}^T) - u_2(q_j)|
    \cot \theta_{i,T},
\end{align}
where the inequality follows from \eqref{eqn:A1:Keta} of Assumption \ref{A1:elliptic}
and \eqref{eqn:rcd002}. For interior nodes $q_i,q_j \in \cQ$, we then have 
\begin{align}\label{eqn:pij005a}
\left| \int_{\omega_{ij}} \vec{b}(x) \tdot \grad \varphi_j \varphi_i \, \d x \right|
 \le \f{K_\eta}{6}\left( 
     \delta_{ij}^e(u_2)( \cot \theta_{ij}^+ + \cot \theta_{ij}^-) + 
      \delta_{\omega_{ij}}^{(i)}(u_2)
      \sum_{T \in \omega_{ij}} \cot \theta_{i,T} \right).
\end{align}
Applying \eqref{eqn:pij002a}, \eqref{eqn:pij003a} and \eqref{eqn:pij005a}
to \eqref{eqn:assemA}, it holds for each $i,j$ for with $q_i, q_j \in \cQ$ that
\begin{align}\label{eqn:pij_006}
a_{ji} & \le \f 1 2\left(-k_\alpha + \f{K_\eta}{3}\delta^e_{ij}(u_2) \right)
(\cot \theta_{ij}^+ + \cot \theta_{ij}^-)
+ \f 1 2\sum_{T \in \omega_{ij}}\!\! \left(\f {K_\eta}{3}\delta_{\omega_{ij}}^{(i)}(u_2) 
\cot \theta_{i,T} + \f{G_\eta|T|}{6} \right)
\nonumber \\
& \le \f 1 2 \left\{ -k_\alpha \beta_m 
 + \f{K_\eta \beta_M}{3}\delta_{\omega_{ij}}(u_2)
 + \f{G_\eta|T_\cT|}{3} \right\},
\end{align}
where the last inequality follows from the application of both angle conditions
\eqref{eqn:opp-angle-cond} and \eqref{eqn:full-angle-cond} from Assumption \ref{A2:mesh}.
If either $q_i$ or $q_j$ lies on the boundary $\pa \Omega$, then the contribution 
of each term is zero, and $a_{ij} = 0$.
The conclusion $a_{ij} \le 0$ then follows under the condition \eqref{eqn:pij-cond}.
\end{proof}

In the next lemma, we show the diagonal entries of $A$ are positive, under the
given condition which bounds the difference of nodal values connecting each edge in the
mesh. The local condition \eqref{eqn:pii-cond} for each $a_{ii}$ to be positive
is weaker than \eqref{eqn:pij-cond}, used above to determine the off-diagonal entries 
of $A$ are nonpositive, implying that matrix $A$ is an $L$-matrix as well as a 
$Z$-matrix. 
\begin{lemma}\label{lem:pii-pos}
Let Assumption \ref{A1:elliptic} and Assumption \ref{A2:mesh} hold. 
Let $A = (a_{ij})$ be given by \eqref{eqn:assemA}.
Then under the condition
\begin{align}\label{eqn:pii-cond}
\delta_{ij}^e(u_2) < \f {3 k_\alpha}{K_\eta},
\end{align}
it holds that $a_{ii} > 0$, for each $i$.
\end{lemma}
\begin{proof}
First consider the diffusion term. Summing integrals over each 
$\omega_{ij} \subset \omega_i$ integrates twice over $\omega_i$.
Applying the identity
$\grad \varphi_i \tdot \grad \varphi_i = -\grad \varphi_j \tdot \grad \varphi_i
- \grad \varphi_{ij}^T \tdot \grad \varphi_i$, over each element $T \in \omega_i$
with nodal indices $q_i, q_j, q_{ij}^T$, and combining like terms to integrate 
each product once per element we have
\begin{align}\label{eqn:pii_001}
\int_{\omega_i} \kappa(x,u_1) \grad \varphi_i \tdot \grad \varphi_i \, \d x
& =  \f 1 2 \sum_{\omega_{ij} \subset \omega_i} 
     \int_{\omega_{ij}} \kappa(x,u_1)\grad \varphi_i \tdot \grad \varphi_i \, \d x
\nonumber \\
& = -\sum_{\omega_{ij} \subset \omega_i} 
     \int_{\omega_{ij}} \kappa(x,u_1)\grad \varphi_j \tdot \grad \varphi_i \, \d x
\nonumber \\
& \ge\f{k_\alpha}{2} \sum_{\omega_{ij} \subset \omega_i} 
    (\cot \theta_{ij}^- + \cot \theta_{ij}^+),
\end{align}
where the last inequality follows from \eqref{eqn:A1:elliptic} and Assumption
\ref{A2:mesh}.
Next, consider the nonsymmetric term.
Summing integrals over each $\omega_{ij} \subset \omega_i$ then combining like terms
as above we have
\begin{align}\label{eqn:pii_002}
\int_{\omega_i}b(x)\grad u_2\tdot \grad \varphi_i \varphi_i \, \d x
& = \f 1 2 \sum_{\omega_{ij} \subset \omega_i} \int_{\omega_{ij}} 
  b(x) \grad u_2\tdot \grad \varphi_i \varphi_i \, \d x
\nonumber \\
& = \sum_{\omega_{ij} \subset \omega_i}(u_2(q_j) - u_2(q_i))
    \int_{\omega_{ij}}\grad \varphi_j \tdot \grad \varphi_i b(x) \varphi_i \, \d x
\nonumber \\
& \ge -\f {K_\eta} {6} \sum_{\omega_{ij} \subset \omega_i} |u_2(q_j) - u_2(q_i)| 
    (\cot \theta_{ij}^+ + \cot \theta_{ij}^-),
\end{align}
where the last inequality follows from \eqref{eqn:A1:Keta}, Assumption \ref{A2:mesh},
and the integration of $\varphi_i$ over each element. 
From \eqref{eqn:rcd002} and \eqref{eqn:pii_002} we have
\begin{align}\label{eqn:pii_003}
\int_{\omega_i} \vec b(x) \tdot \grad \varphi_i \varphi_i \, \d x
\ge -\f{K_\eta}{6}\sum_{\omega_{ij} \subset \omega}
\delta_{ij}^e(u_2)(\cot \theta_{ij}^+ + \cot \theta_{ij}^-).
\end{align}

The lowest order term from \eqref{eqn:assemA} satisfies
$\int_{\omega_i} c(x) \varphi_i ^2 \, \d x  \ge 0,$
for $c(x) \ge 0$ as in \eqref{eqn:rcd003} under the condition \eqref{eqn:A1:Geta}.
Putting together \eqref{eqn:pii_001} and \eqref{eqn:pii_003} 
into \eqref{eqn:assemA}, we have under Assumption \ref{A1:elliptic}
\begin{align}\label{eqn:pii_004}
a_{ii} &\ge \f 1 2  \sum_{\omega_{ij} \subset \omega_i} 
(\cot \theta_{ij}^+ + \cot \theta_{ij}^-)
 \left(-\f {K_\eta} 3 \delta_{ij}^e(u_2) + k_\alpha \right)\nonumber\\
 &\ge \f{\beta_m}{2} \left(-\f {K_\eta} 3 \delta_{ij}^e(u_2) + k_\alpha \right),
\end{align}
from which the result follows under condition \eqref{eqn:pii-cond}.
\end{proof}

By the conditions of the previous two lemmas, and the computations of the next lemma,
the matrx $A$ can be seen to have positive row-sums for each index $i$ 
such that vertex $q_i$ neighbors the Dirichlet boundary. 
In the next lemma, we constuct a diagonal matrix $D_\eps'$ for which the positivity
of row-sums of $AD_\eps'$ is extended to indices $j$ such that $q_j$ 
neighbors a vertex which neighbors $\Gamma_D$.  
This is not sufficient to establish the comparison result 
which requires strict diagonal dominance of $AD$ for some $D$. However, it indicates
how to construct a diagonal matrix $D$ with positive diagonal entries such that the 
positivity of row-sums of $AD$ 
can be extended to rows corresponding to interior vertices. We present this lemma first
because it contains the main ideas and estimates, and is simpler in its presentation.  
Lemma \ref{lem:strict-diag-dom} then contains the full
construction of a diagonal matrix $D_\eps$ for which $AD_\eps$ is strictly diagonally 
dominant.
\begin{lemma}\label{lem:diag-dom}
Let Assumption \ref{A1:elliptic} and Assumption \ref{A2:mesh} hold. 
Let $A = (a_{ij})$ be given by \eqref{eqn:assemA}, and let $A^T = \alpha_{ij}$.
Assume the conditions of Lemma \ref{lem:pij-nonpos} and \ref{lem:pii-pos} hold true, and
for some $\bar \eps > 0$ it holds that 
\begin{align}\label{eqn:dd-cond}
\f{K_\eta \beta_M \delta_{\omega_{ij}}(u_2)}{3\beta_m} 
< k_\alpha - \bar \eps.
\end{align}
Let $D_\eps'$ be the diagonal matrix with entries $d_i$ given by
\begin{align}\label{eqn:Ddef}
d_i = \left\{\begin{array}{ll} 1, & ~\text{ if }~ \cQ_i = \overline \cQ_i, \\ 
                             d_\eps < 1, & ~\text{ otherwise. }  \end{array}\right.
\end{align}
Then, as $d_\eps \goto 1$ but $d_\eps < 1$, the matrix $A^TD_\eps'$ is 
diagonally dominant, and has positive row-sums for each index $i$ for which 
$q_i$ neighbors $\Gamma_D$, or has a neighbor that does.
\end{lemma}
The idea behind the construction of $D_\eps'$ relates to the stiffness matrix $S$ for 
the Laplacian.  For any row $i$ such that corresponding vertex $q_i$ has a neighbor
on the Dirichlet boundary, the row-sum of $S$ 
is positive, and otherwise zero. This leaves enough room to scale down the columns 
corresponding to vertices with neighbors on the Dirichlet boundary so that all 
row-sums containing nonzero terms from these columns are also positive.  
The construction of $D_\eps$ so that $A^TD_\eps$ is strictly diagonally 
dominant follows from scaling each column closer to one as its distance from the 
Dirichlet boundary increases, and is addressed Lemma \ref{lem:strict-diag-dom}.
The monotonicity of $A$ follows from the monotonicity of $A^T$.
\begin{proof}
By the established positivity of the diagonal elements, and nonpositivity of the
off-diagonal elements, the row-$i$ requirement for diagonal dominance of 
$A^TD_\eps'$ is
\begin{align}\label{eqn:dd_001}
d_i \alpha_{ii} + \sum_{j = 1, j \ne i}^n d_j \alpha_{ij} \ge 0,
\end{align}
where the inequality must be strict for at least one index $i$.
By slight abuse of notation, let $j \in \cQ_i$ mean index $j$ such that $q_j \in \cQ_i$.
Let $n = \card(\cQ)$, the number of mesh degrees of freedom.
Expanding \eqref{eqn:dd_001} by \eqref{eqn:assemA} and rearranging terms yields
\begin{align}\label{eqn:dd_002}
d_i \alpha_{ii} + \sum_{j = 1, j \ne i}^n d_j \alpha_{ij} & = 
d_i \int_{\omega_i} \kappa(x,u_1) \grad \varphi_i \tdot \grad \varphi_i \, \d x
+ \sum_{j \in \cQ_i} d_j \int_{\omega_{ij}} \kappa(x,u_1) 
\grad \varphi_i \tdot \grad \varphi_j \, \d x
\nonumber \\
& + d_i \int_{\omega_i} \vec b(x) \tdot  \grad \varphi_i \varphi_i \, \d x
+ \sum_{j \in \cQ_i} d_j \int_{\omega_{ij}} \vec b(x) \tdot 
\grad \varphi_j \varphi_i \, \d x
\nonumber \\
& + d_i \int_{\omega_i} c(x) \varphi_i^2 \, \d x
+ \sum_{j \in \cQ_i} d_j \int_{\omega_{ij}} c(x) \varphi_j \varphi_i \, \d x.
\end{align}

Similarly to the calculations of \eqref{eqn:pij002a} and \eqref{eqn:pii_001}
the contribution from the first line of \eqref{eqn:dd_002} is 
\begin{align}\label{eqn:dd_003}
&d_i \int_{\omega_i} \kappa(x,u_1) \grad \varphi_i \tdot \grad \varphi_i \, \d x
+ \sum_{j \in \cQ_i} d_j \int_{\omega_{ij}} \kappa(x,u_1) 
\grad \varphi_i \tdot \grad \varphi_j \, \d x
\nonumber \\
& = -d_i\sum_{j \in \overline \cQ_i\setminus \cQ_i} 
  \int_{\omega_{ij}} \kappa(x,u_1) \grad \varphi_i \tdot \grad \varphi_j \, \d x 
+ \sum_{j \in \cQ_i} (d_j-d_i) 
  \int_{\omega_{ij}} \kappa(x,u_1) \grad \varphi_i \tdot \grad \varphi_j \, \d x. 
\end{align}
It is noted that the contribution from the second term
on the left of \eqref{eqn:dd_002} is
restricted to $j \in \cQ_i$, the mesh degrees of freedom.

Recalling that $\vec b(x) = b(x) \nabla u_{2},$ the contribution from the second line 
of \eqref{eqn:dd_002} is 
\begin{align}\label{eqn:dd_004}
& d_i \int_{\omega_i} \vec b(x) \tdot  \grad \varphi_i \varphi_i \, \d x
+ \sum_{j \in \cQ_i} d_j \int_{\omega_{ij}} \vec b(x) \tdot 
\grad \varphi_j \varphi_i \, \d x
\nonumber \\
& = d_i \sum_{j \in \overline \cQ_i \setminus \cQ_i} 
  (u_2(q_j) - u_2(q_i)) \int_{\omega_{ij}} \grad \varphi_j \cdot \grad \varphi_i
   b(x) \varphi_i \, \d x 
\nonumber \\
& + \sum_{j \in \cQ_i} (d_j - d_i)(u_2(q_i) - u_2(q_j)) \int_{\omega_{ij}}
  \grad \varphi_i \tdot \grad \varphi_j b(x) \varphi_i \, \d x
\nonumber \\
& + \sum_{j \in \cQ_i} d_j \sum_{T \in \omega_{ij}}(u_2(q_{ij}^T) - u_2(q_j)) 
  \int_{T}\grad \varphi_{ij}^T \tdot \grad \varphi_j b(x) \varphi_i \, \d x.
\end{align}
Let $\cK_{ij}^T \coloneqq (u_2(q_{ij}^T) - u_2(q_j)) 
  \int_{T}\grad \varphi_{ij}^T \tdot \grad \varphi_j b(x) \varphi_i \, \d x.$
The last line of \eqref{eqn:dd_004} can be expanded as 
\begin{align}\label{eqn:dd_005}
&  \left( d_i\sum_{j \in \overline \cQ_i}\sum_{T \in \omega_{ij}} \cK_{ij}^T\right)
 + \left(\sum_{j \in  \cQ_i} (d_j - d_i) \sum_{T \in \omega_{ij}} \cK_{ij}^T \right)
 - \left( d_i \sum_{j \in \overline \cQ_i \setminus \cQ_i} \sum_{T \in \omega_{ij}} 
   \cK_{ij}^T \right)
\nonumber \\
& = \left(\sum_{j \in  \cQ_i} (d_j - d_i) \sum_{T \in \omega_{ij}} \cK_{ij}^T \right)
  - \left( d_i \sum_{j \in \overline \cQ_i \setminus \cQ_i} \sum_{T \in \omega_{ij}} 
    \cK_{ij}^T \right),
\end{align}
as the first term in the left of \eqref{eqn:dd_005} is zero because the $\cK_{ij}^T$ terms 
cancel pairwise when summed over the entire patch $\overline \cQ_i$. 

The contribution from the third line of \eqref{eqn:dd_002} can be written as 
\[
\f {d_i}{2} \sum_{j \in \overline \cQ_i \setminus \cQ_i} 
\int_{\omega_{ij}} c(x) \varphi_i^2 \, \d x +  
\sum_{j \in \cQ_i} \int_{\omega_{ij}} c(x) \varphi_i
  \left(\f{d_i}{2}\varphi_i + d_j \varphi_j \right) \d x \ge 0.
\]
This term is clearly nonnegative and need not be further considered.
Recombining the remaining terms of \eqref{eqn:dd_003}, \eqref{eqn:dd_004} and 
\eqref{eqn:dd_005} into \eqref{eqn:dd_002} we have
\begin{align}\label{eqn:dd_006}
&d_i \alpha_{ii} + \sum_{j = 1, j \ne i}^n d_j \alpha_{ij} 
\nonumber \\
& \ge \sum_{j \in \overline \cQ_i \setminus \cQ_i} d_i\Bigg\{ 
     \int_{\omega_{ij}} \left(\kappa(x,u_1) - (u_2(q_j) - u_2(q_i)) b(x) \varphi_i)
    \right) \grad \varphi_i \tdot \grad \varphi_j \, \d x 
\nonumber \\
&  - \sum_{T \in \omega_{ij}} 
  (u_2(q_{ij}^T) - u_2(q_j)) \int_T 
    \grad \varphi_{ij}^T \tdot \grad \varphi_j b(x) \varphi_i \, \d x \Bigg\}
\nonumber \\
& + \sum_{j \in \cQ_i} (d_j - d_i) \Bigg(
  \int_{\omega_{ij}} (\kappa(x,u_1) + 
  (u_2(q_i) - u_2(q_j))b(x) \varphi_i)\grad \varphi_i \tdot \grad \varphi_j \, \d x
\nonumber \\
&  + \sum_{T \in \omega_{ij}} (u_2(q_{ij}^T) - u_2(q_j)) 
  \int_{T}\grad \varphi_{ij}^T \tdot \grad \varphi_j b(x) \varphi_i \, \d x \Bigg).
\end{align}
The last two lines of \eqref{eqn:dd_006} are controlled
by the hypothesis \eqref{eqn:dd-cond},  {\em cf.} \eqref{eqn:pij005}.
In particular
\begin{align}\label{eqn:dd_006a}
\cJ_{ij} & \coloneqq \int_{\omega_{ij}} (\kappa(x,u_1) + 
(u_2(q_i) - u_2(q_j))b(x) \varphi_i)\grad \varphi_i \tdot \grad \varphi_j \, \d x
\nonumber \\
& + \sum_{T \in \omega_{ij}} (u_2(q_{ij}^T) - u_2(q_j)) 
  \int_{T}\grad \varphi_{ij}^T \tdot \grad \varphi_j b(x) \varphi_i \, \d x
\nonumber \\
& \le \f 1 2 \left( -k_\alpha \beta_m 
 + \f{K_\eta \beta_M\delta_{\omega_{ij}}(u_2)}{3} \right) < -\f{\beta_m\bar \eps}{2} < 0.
\end{align}
It is also important to note that taking into consideration the upper bound on 
$\kappa$ given in \eqref{eqn:A1:elliptic}, we have finite numbers 
$\cJ_L, \cJ^U$ for which
\begin{align}\label{eqn:dd_006b}
\f{\beta_m \bar \eps}{2} \le \cJ_L \le |\cJ_{ij}| \le \cJ^U.
\end{align} 
Applying \eqref{eqn:dd_006a}, the angle conditions \eqref{eqn:opp-angle-cond} and
\eqref{eqn:full-angle-cond}, and the bounds on the data as above, 
inequality \eqref{eqn:dd_006} implies
\begin{align}\label{eqn:dd_007}
&d_i \alpha_{ii} + \sum_{j = 1, j \ne i}^n d_j \alpha_{ij}
\nonumber \\ 
& \ge \f 1 2\sum_{j \in \overline \cQ_i \setminus \cQ_i} \left\{ 
    d_i \left(k_\alpha - \f{K_\eta \delta_{ij}^e(u_2)}{3} \right) 
    (\cot \theta_{ij}^+ + \cot \theta_{ij}^-)  
  -d_j \f{K_\eta \delta_{\omega_{ij}}^{(i)}}{3}  
    \sum_{T \in \omega_{ij} }  \cot \theta_i^T  \right\}
\nonumber \\
& +  \sum_{j \in \cQ_i} (d_j - d_i) \cJ_{ij}. 
\end{align}

In the case that $\overline \cQ_i \setminus \cQ_i =  \emptyset$, meaning
vertex $q_i$ has no neighbors on $\Gamma_D$, the
first term on the RHS of \eqref{eqn:dd_007} does not appear, and $d_i = 1$.
Either $d_j = 1$, in which case the second term
on the RHS of \eqref{eqn:dd_007}
is zero, or $d_j = d_\eps < 1$, for some $j \in \cQ_i$, meaning $(d_j - d_i) < 0$, 
and the  second term is positive.

In the case that $\overline \cQ_i \setminus \cQ_i \ne \emptyset$, the 
vertex $q_i$ has at least one neighbor on the Dirichlet boundary, and $d_i = d_\eps$.
The second term of \eqref{eqn:dd_007} is either zero or negative and goes to zero, 
because applying the upper-bound on $|\cJ_{ij}|$ from \eqref{eqn:dd_006b}, we have 
$|\cJ_{ij}(d_j - d_i)| \le \cJ^U (1-d_\eps) \goto 0$ as $d_\eps \goto 1$. 
Moreover, the first term in brackets is strictly positive under condition 
\eqref{eqn:dd-cond}. Specifically, for $d_i = d_\eps$ we have
\begin{align}\label{eqn:dd_008}
d_i a_{ii} + \sum_{j = 1, j \ne i}^n d_j a_{ij} 
 &\ge \f {d_i \beta_m}{2} \sum_{j \in \overline \cQ_i \setminus \cQ_i}
    \left(k_\alpha - \f{K_\eta \beta_M\delta_{ij}(u_2)}{3\beta_m} \right) 
 + \sum_{j \in \cQ_i} (d_j - d_i) \cJ_{ij} 
\nonumber \\
&\ge \f{d_\eps \beta_m \bar \eps}{2} + \sum_{j \in \cQ_i}(1-d_\eps) \cJ^U,
\end{align}
which is clearly positive for any fixed $\bar \eps > 0$ as $d_\eps \goto 1$, minding 
the angle condition \eqref{eqn:full-angle-cond} strictly bounds the number of 
neighbors for any vertex $\cQ_i$, so the sum on the right has a maximum $m$ 
number of terms.

These arguments together show all row-sums are nonnegative; 
and, row-sums are strictly positive for vertices with neighbors on 
$\Gamma_D$, or that have neighbors that neighbor $\Gamma_D$.
As there is at least one Dirichlet node, the conclusion follows.
\end{proof}

We now construct a slightly more complicated diagonal matrix $D_\eps$ 
to establish strict diagonal dominance of $A^TD_\eps$.  
The proof is similar to Lemma \ref{lem:diag-dom}, with some additional 
arguments.  Instead of defining the diagonal elements $d_i$ of $D_\eps$ corresponding
to nodes without neighbors on the Dirichlet boundary to be unity, they are defined 
as an increasing sequence based on their distance from $\Gamma_D$.  This
prevents the case of the zero row-sum corresponding to vertices sufficiently far 
from the boundary as in Lemma \ref{lem:diag-dom}. 
First, we require a notion of distance from the boundary.

\begin{definition}\label{def:dist}
Let $p_i$ denote the length of the shortest path to a neighborhood of 
the Dirichlet boundary from vertex $q_i$.  
In particular, if $\overline \cQ_i \setminus \cQ_i \ne \emptyset$, then $p_i = 0$.  
Otherwise, $p_i$ is defined to be the least number of edges traversed
between $q_i$ and any vertex $q_j$ with $p_j = 0$.
\end{definition}

\begin{lemma}\label{lem:strict-diag-dom}
Let Assumption \ref{A1:elliptic} and Assumption \ref{A2:mesh} hold. 
Let $A = (a_{ij})$ be given by \eqref{eqn:assemA}, and let $A^T = \alpha_{ij}$.
Assume the conditions of Lemma \ref{lem:pij-nonpos} and \ref{lem:pii-pos} hold true, and
for some $\bar \eps > 0$ it holds that 
\begin{align}\label{eqn:s-dd-cond}
\f{K_\eta \beta_M\delta_{\omega_{ij}}(u_2)}{3\beta_m} 
< k_\alpha - \bar\eps,
\end{align}
as in Lemma \ref{lem:diag-dom}.
Let $D_\eps$ be the diagonal matrix with entries $d_i$ given by
\begin{align}\label{eqn:epsDdef}
d_i = 1 - \eps_{p_i}, \quad 
\eps_{p_i} = \eps_{p_i-1} - r^{p_i-1}\delta_0, ~p_i \ge 1,
\end{align}
for $p_i$ given by Definition \ref{def:dist},
fixed $0 < \eps_0 < 1$, and $0 < r, \delta_0 < 1$, to be defined below.
Then, matrix $A^TD_\eps$ is strictly diagonally dominant, and
the condition \eqref{eqn:s-dd-cond} relaxes to the condition 
\eqref{eqn:pij-cond} for $A$ to be a $Z$-matrix, as $\bar \eps \goto 0$. 
\end{lemma}

\begin{proof}
First it is noted that the sequence $\{ \eps_{j}\}$ from \eqref{eqn:Ddef}
is a strictly decreasing sequence.  By summing the geometric terms in 
\eqref{eqn:epsDdef}, we also see the sequence $\{ \eps_{j}\}$ is strictly positive
if $\eps_0 > \delta_0/(1-r)$, which will be assured as $\bar \eps \goto 0$ for fixed
$\eps_0$. As a result, 
the coefficients $d_i$ are ordered by the distance of each $q_i$ to the boundary 
with respect to Definition \ref{def:dist}, and 
$d_i \goto 1 +  \delta_0/(1-r) - \eps_0,$ for increasing $p_i$.

Similarly to \eqref{eqn:dd_001}, we require for each row $i$ of 
the product $AD_\eps$ that
\begin{align}\label{eqn:sdd_001}
d_i \alpha_{ii} + \sum_{j = 1, j \ne i}^n d_j \alpha_{ij} > 0,
\end{align}
where the row-sum is given explicity by \eqref{eqn:dd_002}. Each line of 
\eqref{eqn:dd_002} is now considered with respect to the membership of each
$q_j \in \overline \cQ_i$ in one of three sets.

Define the sets 
\begin{align}\label{def:Qp}
\cQ_i^p \coloneqq \{ q_j \in \overline \cQ_i ~\rest~ p_j = p\}.
\end{align}
We also denote $\cQ_i^{-1} = \overline \cQ_i \setminus \cQ_i$. 
A first key observation for the following analysis is for each vertex 
$q_j \in \overline \cQ_i$, 
$q_j$ is in exactly one of $\cQ_i^{p_i-1}, \cQ_i^{p_i}, \cQ_i^{p_i + 1}$.
A second key observation is at least one $q_j$ in $\overline \cQ_i$ is in 
$\cQ_i^{p_i-1}$, meaning at least one neighbor of $q_i$ is closer in the sense of 
Definition \ref{def:Qp} to $\Gamma_D$.
We now partition the terms of \eqref{eqn:dd_002} into sums over each of these three
sets. As before, the contribution to the total sum from the terms in the last
line of \eqref{eqn:dd_002} is strictly nonnegative, so we only consider the terms
in the first two lines.

Again, let $j \in \cQ_i^p$ mean index $j$ for which $q_j \in \cQ_i^p$.  
For simplicity of notation, let $p$ denote $p_i$.
If $p = 0$, meaning vertex $q_i$ neighbors the Dirichlet boundary, 
the contribution from the first line on the RHS of \eqref{eqn:dd_002} is then
\begin{align}\label{eqn:sdd_003p0}
 -d_i \sum_{j \in \cQ_i^{-1}} \int_{\omega_{ij}} \kappa(x,u_1) 
    \grad \varphi_i \grad \varphi_j \, \d x +
  \sum_{j\in \cQ_i} (d_j - d_i)\int_{\omega_{ij}}
  \kappa(x,u_1) \grad \varphi_i \grad \varphi_j \, \d x.
\end{align}
For $p > 0$, meaning vertices $q_i$ without neighbors on $\Gamma_D$, we have
\begin{align}\label{eqn:sdd_003}
  \sum_{j\in \cQ_i} (d_j - d_i)\int_{\omega_{ij}}
  \kappa(x,u_1) \grad \varphi_i \grad \varphi_j \, \d x.
\end{align}

For $p = 0$, the contribution from the second line of \eqref{eqn:dd_002} can be written
as 
\begin{align}\label{eqn:sdd_004p0}
&  d_i \sum_{j \in \cQ_i^{-1}} 
  (u_2(q_j) - u_2(q_i)) \int_{\omega_{ij}} \grad \varphi_j \cdot \grad \varphi_i
   b(x) \varphi_i \, \d x 
\nonumber \\
& + \sum_{j \in \cQ_i} (d_j - d_i)(u_2(q_i) - u_2(q_j)) \int_{\omega_{ij}}
  \grad \varphi_i \tdot \grad \varphi_j b(x) \varphi_i \, \d x
\nonumber \\
& + \sum_{j \in \cQ_i} d_j \sum_{T \in \omega_{ij}} (u_2(q_{ij}^T) - u_2(q_j)) 
    \int_{T}\grad \varphi_{ij}^T \tdot \grad \varphi_j b(x) \varphi_i \, \d x.
\end{align}

As in \eqref{eqn:dd_005}, the third line of \eqref{eqn:sdd_004p0} can be expanded as
\begin{align}\label{eqn:sdd_005p0}
& \sum_{j \in  \cQ_i} (d_j - d_i) \sum_{T \in \omega_{ij}}
(u_2(q_{ij}^T) - u_2(q_j)) 
  \int_{T}\grad \varphi_{ij}^T \tdot \grad \varphi_j b(x) \varphi_i \, \d x
\nonumber \\
&  - d_i \sum_{j \in \cQ_i^{-1}} \sum_{T \in \omega_{ij}} (u_2(q_{ij}^T) - u_2(q_j)) 
  \int_{T}\grad \varphi_{ij}^T \tdot \grad \varphi_j b(x) \varphi_i \, \d x.
\end{align}

For $p > 0$, the contribution from the second line of \eqref{eqn:dd_002} can be
written as 
\begin{align}\label{eqn:sdd_004}
&  \sum_{j \in \cQ_i} (d_j - d_i)(u_2(q_i) - u_2(q_j)) \int_{\omega_{ij}}
  \grad \varphi_i \tdot \grad \varphi_j b(x) \varphi_i \, \d x
\nonumber \\
& + \sum_{j \in \cQ_i} d_j \sum_{T \in \omega_{ij}}(u_2(q_{ij}^T) - u_2(q_j)) 
    \int_{T}\grad \varphi_{ij}^T \tdot \grad \varphi_j b(x) \varphi_i \, \d x.
\end{align}
As in \eqref{eqn:dd_005} and \eqref{eqn:sdd_005p0}, the second line of 
\eqref{eqn:sdd_004} can then be written as 
\begin{align}\label{eqn:sdd_005}
& \sum_{j \in  \cQ_i} 
  (d_j - d_i) \sum_{T \in \omega_{ij}}(u_2(q_{ij}^T) - u_2(q_j)) 
  \int_{T}\grad \varphi_{ij}^T \tdot \grad \varphi_j b(x) \varphi_i \, \d x.
\end{align}

For the case $p = 0$, applying expansions \eqref{eqn:sdd_003p0}, \eqref{eqn:sdd_004p0} and
\eqref{eqn:sdd_005p0} to \eqref{eqn:dd_002} we have
\begin{align}\label{eqn:sdd_006p0}
&d_i \alpha_{ii} + \sum_{j = 1, j \ne i} d_j \alpha_{ij} 
\nonumber \\
& \ge \sum_{j \in \cQ_i^{-1}} \Bigg\{ 
    -d_i\int_{\omega_{ij}} \left(\kappa(x,u_1) - (u_2(q_j) - u_2(q_i)) b(x) \varphi_i)
    \right) \grad \varphi_i \tdot \grad \varphi_j \, \d x 
\nonumber \\
& - d_j\sum_{T \in \omega_{ij} } (u_2(q_{ij}^T) - u_2(q_j)) \int_T b(x) \varphi_i 
    \grad \varphi_{ij}^T \tdot \grad \varphi_j \, \d x \Bigg\}
+  \sum_{j \in \cQ_i} (d_j - d_i)\cJ_{ij}, 
\end{align}
with $\cJ_{ij}< 0$ given by \eqref{eqn:dd_006a}.
The strict positivity of \eqref{eqn:sdd_006p0} follows as in the previous lemma, 
up to the last term of \eqref{eqn:sdd_006p0}.
To control the additional term, 
let $m$ be the maximum number of neighbors of any given vertex. A fixed
maximum $m$ is implied by the smallest-angle condition \eqref{eqn:full-angle-cond}.
Minding $\{\eps_j\}$ is a decreasing sequence,
the last term in \eqref{eqn:sdd_006p0} is bounded as
\begin{align}\label{eqn:sdd_007p0}
\sum_{j \in \cQ_i}(d_j - d_i) \cJ_{ij} 
 = \sum_{j \in \cQ_i^1}(\eps_0 - \eps_1) \cJ_{ij}
 \ge -(m-1)(\eps_0 - \eps_1) \cJ^U.
\end{align}
The sum over $\cQ_i^{-1}$ contains at least one term, and similarly to
\eqref{eqn:dd_008}, the estimates \eqref{eqn:sdd_006p0} and \eqref{eqn:sdd_007p0} with
the condition \eqref{eqn:s-dd-cond}, and $d_i = 1-\eps_0$, imply
\begin{align}\label{eqn:dd_008p0}
d_i \alpha_{ii} + \sum_{j \ne i} d_j \alpha_{ij} 
 &\ge \f {(1 - \eps_0)\beta_m}{2} \sum_{j \in \cQ_i^{-1}}
    \left(k_\alpha - \f{K_\eta \beta_M\delta_{ij}(u_2)}{3\beta_m} \right) 
 + \sum_{j \in \cQ_i^1} (d_j - d_i) \cJ_{ij} 
\nonumber \\
&\ge \f{(1 - \eps_0)\bar \eps \beta_m}{2} - (m-1)(\eps_0 - \eps_1) \cJ^U.
\end{align}
From \eqref{eqn:epsDdef}, $\eps_0 - \eps_1 = \delta_0$, so setting 
\begin{align}\label{eqn:sdd_008p0a}
\delta_0 \coloneqq \f{\bar \eps \beta_m(1 - \eps_0)}{2 m\cJ^U},
\end{align}
forces the row-sum in \eqref{eqn:dd_008p0} to be strictly positive for the case $p = 0$.

For the case $p > 0$, applying expansions \eqref{eqn:sdd_003}, \eqref{eqn:sdd_004}
and \eqref{eqn:sdd_005} to \eqref{eqn:dd_002}, and $\cJ_{ij}< 0$ from 
\eqref{eqn:dd_006a}, we have
\begin{align}\label{eqn:sdd_006}
d_i \alpha_{ii} + \sum_{j = 1, j \ne i} d_j \alpha_{ij} 
& = \sum_{j \in \cQ_i^{p-1}} (d_j - d_i)\cJ_{ij} 
  +  \sum_{j \in \cQ_i^{p+1}} (d_j - d_i)\cJ_{ij}
\nonumber \\ 
& = \sum_{j \in \cQ_i^{p-1}} (\eps_{p} - \eps_{p-1})\cJ_{ij} 
  +  \sum_{j \in \cQ_i^{p+1}} (\eps_p - \eps_{p+1})\cJ_{ij}
\nonumber \\
& \ge  (\eps_{p-1} - \eps_{p}) \cJ_L
  -  (m-1) (\eps_p - \eps_{p+1}) \cJ^U,
\end{align}
where the first sum must be nonempty because at least one vertex $q_j \in \cQ_i$ must
be closer to the boundary than $q_i$.
From the construction \eqref{eqn:epsDdef}, \eqref{eqn:sdd_006} then implies
\begin{align}\label{eqn:sdd_007}
d_i \alpha_{ii} + \sum_{j = 1, j \ne i} d_j \alpha_{ij} 
& \ge  r^{p-1}\delta_0 \cJ_L
  -  (m-1)r^p\delta_0 \cJ^U> 0,
\end{align}
for $r < \cJ_L/((m-1) \cJ^U)$.
In particular, inequality \eqref{eqn:sdd_007} is satisfied for
\begin{align}\label{eqn:sdd_008}
r \coloneqq \f{ \cJ_L}{m \cJ^U} < 1.
\end{align}
It is finally noted for the sequence $\{\eps_i\}$ given  by \eqref{eqn:epsDdef}, 
\[
\eps_i > \eps_0 - \f{\delta_0}{1-r} 
= \eps_0 \left(1 + \f{\bar \eps \beta_m}{2(m\cJ^U - \cJ_L)} \right) - 
\f{\bar \eps \beta_m}{2(m\cJ^U - \cJ_L)} > 0,
\]
for $\bar \eps$ small enough.

These arguments together show the matrix $A^TD_\eps$ is strictly diagonally dominant, for 
diagonal matrix $D_\eps$ defined by \eqref{eqn:epsDdef} with 
$\delta_0$ given by \eqref{eqn:sdd_008p0a} and $r$ given by
\eqref{eqn:sdd_008}.  The remainder of the result follows by sending $\bar \eps \goto 0$.
\end{proof}

We summarize our results in a corollary.
\begin{corollary}Assume satisfaction of the hypotheses and conditions 
of Lemma \ref{lem:strict-diag-dom}.  
Then $A = (a_{ij})$ given by \eqref{eqn:assemA} is monotone. 
Moreover if the conditions given on the supersolution $u_2$ are satisfied by any 
solution $u \in \cV$ to \eqref{eqn:QLPDEdiscrete}, the solution is unique.
\end{corollary}
\begin{proof}
Apply the conclusions of Lemma \ref{lem:strict-diag-dom} to Theorem \ref{thm:FiPe} to
establish the monotonicity of $A^T$ which directly implies the monotonicity of $A$. 
The second conclusion follows by Theorem \ref{thm:comparison}.
\end{proof}

As a final remark, the conditions given in Lemma \ref{lem:strict-diag-dom} which
imply the comparison theorem and uniqueness of the solution, improve the conditions
found in previous work by the authors.
\begin{remark}
To ilustrate this, consider an equilateral mesh.
Then the minimum ratio of sines is equal to one, and the cosine of each angle is 
$c_T = 1/2$.  Then the condition for unqiueness found in \cite[Theorem 3.4]{PoZh17a} 
for the 2D case reduces to 
\[
\delta_T(u) < \f {3 k_\alpha}{14 K_\eta}.
\]
To put the result in the current notation, 
the Lipschitz constant $L_0$ is taken as $K_\eta$.
The conditions found in this investigation
for the same problem \eqref{eqn:QLPDEstrong} without the lower order term on an 
equilateral mesh are found by noting $\beta_m = 1/\sqrt 3$, and $\beta_M = 4/\sqrt 3$.
Then as $\eps \goto 0$, the requirement is
\[
\delta_T(u) < \f{3 k_\alpha}{4 K_\eta}, 
\]
which improves the constant by a factor of $7/2$.
\end{remark}
\section{Conclusion}
In this article, we established a discrete comparison theorem for 
\eqref{eqn:QLPDEstrong}, a quasilinear PDE with a solution-dependent lower order term.
We established sufficient local and global conditions for the monotonicity of the 
assembled coefficient matrix for the PDE corresponding to the difference of two 
solutions. The monotonicity then implies uniqueness of the finite element solution 
under the given conditions.
This argument was seen to relax the angle conditions to allow some right triangles in
the mesh, so long as the sum of angles opposite each edge remains bounded below $\pi$.
Considering the elements of the assembled matrix rather than the integral over 
each individual element further allows an improved local condition on the 
maximum difference between neighboring nodal values.  As in previous work,
we find the mesh should be globally fine if the PDE contains a lower-order 
solution-dependent nonlinearity.
Otherwise, the mesh is required to be fine where the gradient of the solution is steep.

% -- ----------------------------------------------------------------------------
%\section{Acknowledgments}
%   \label{sec:ack}
% -- ----------------------------------------------------------------------------

%\bibliographystyle{siam}
\bibliographystyle{abbrv}
\bibliography{refsTRN,M}

\begin{thebibliography}{10}

\bibitem{AbVi12}
A.~Abdulle and G.~Vilmart.
\newblock A priori error estimates for finite element methods with numerical
  quadrature for nonmonotone nonlinear elliptic problems.
\newblock {\em Numer.\ Math.}, 121(3):397--431, 2012.

\bibitem{AnCh96a}
N.~Andr\'e and M.~Chipot.
\newblock Uniqueness and nonuniqueness for the approximation of quasilinear
  elliptic equations.
\newblock {\em SIAM J.\ Numer.\ Anal.}, 33(5):1981--1994, 1996.

\bibitem{Bouchon07}
F.~Bouchon.
\newblock Monotonicity of some perturbations of irreducibly diagonally dominant
  {M}-matrices.
\newblock {\em Numer.\ Math.}, 105(4):591--601, 2007.

\bibitem{BrKoKr08}
J.~H. Brandts, S.~Korotov, and M.~K{\u r}{\'i}{\u z}ek.
\newblock The discrete maximum principle for linear simplicial finite element
  approximations of a reaction-diffusion problem.
\newblock {\em Linear Algebra Appl.}, 429(10):2344--2357, 2008.
\newblock Special Issue in honor of Richard S. Varga.

\bibitem{CarlLeMontreanu}
S.~Carl, V.~K. Le, and D.~Motreanu.
\newblock {\em Nonsmooth variational problems and their inequalities:
  Comparison principles and applications.}
\newblock Springer monographs in mathematics. New York : Springer
  Science+Business Media, 2007.

\bibitem{Collatz.L1966}
L.~Collatz.
\newblock {\em Functional analysis and numerical mathematics}.
\newblock Translated from the German by Hansj\"org Oser. Academic Press, New
  York, 1966.

\bibitem{DoDu75}
J.~Douglas and T.~Dupont.
\newblock A {G}alerkin method for a nonlinear {D}irichlet problem.
\newblock {\em Mathematics of Computation}, (131):689--696, 1975.

\bibitem{DoDuSe71}
J.~Douglas, T.~Dupont, and J.~Serrin.
\newblock Uniqueness and comparison theorems for nonlinear elliptic equations
  in divergence form.
\newblock {\em Arch. for Ration. Mech. Anal.}, 42(3):157--168, 1971.

\bibitem{DrDuSc04}
A.~Dr{\u a}g{\u a}nescu, T.~F. Dupont, and L.~R. Scott.
\newblock Failure of the discrete maximum principle for an elliptic finite
  element problem.
\newblock {\em Math.\ Comp.}, 74(249):1--23, 2005.

\bibitem{FiPe62}
M.~Fiedler and V.~Pt\'ak.
\newblock On matrices with non-positive off-diagonal elements and positive
  principal minors.
\newblock {\em Czechoslovak Mathematical Journal}, 12(3):382--400, 1962.

\bibitem{GilbargTrudinger}
D.~Gilbarg and N.~S. Trudinger.
\newblock {\em Elliptic partial differential equations of second order.}
\newblock Grundlehren der mathematischen Wissenschaften: 224. Berlin ; New York
  : Springer-Verlag, 1983.

\bibitem{HlKrMa94}
I.~Hlav{\'a}{\u c}ek, M.~K{\u r}{\'i}{\u z}ek, and J.~Mal{\'y}.
\newblock On {G}alerkin approximations of a quasilinear nonpotential elliptic
  problem of a nonmonotone type.
\newblock {\em J. Math. Anal. Appl.}, 184(1):168--189, 1994.

\bibitem{Plemmons77}
R.~J. Plemmons.
\newblock {\em Linear Algebra Appl.}, 18:175--188, 1977.

\bibitem{PoZh17b}
S.~Pollock and Y.~Zhu.
\newblock Discrete comparison principles for quasilinear elliptic {PDE}, 2017.
\newblock Submitted.

\bibitem{PoZh17a}
S.~Pollock and Y.~Zhu.
\newblock Uniqueness of discrete solutions of nonmonotone {PDE}s without a
  globally fine mesh condition, 2017.
\newblock Submitted.

\bibitem{Trudinger74}
N.~S. Trudinger.
\newblock On the comparison principle for quasilinear divergence structure
  equations.
\newblock {\em Arch. for Ration. Mech. and Anal.}, 57(2):128--133, Jun 1974.

\bibitem{Varga.R2000}
R.~S. Varga.
\newblock {\em Matrix iterative analysis}, volume~27 of {\em Springer Series in
  Computational Mathematics}.
\newblock Springer-Verlag, Berlin, expanded edition, 2000.

\bibitem{Windisch.G1989}
G.~Windisch.
\newblock {\em {$M$}-matrices in numerical analysis}, volume 115 of {\em
  Teubner-Texte zur Mathematik [Teubner Texts in Mathematics]}.
\newblock BSB B. G. Teubner Verlagsgesellschaft, Leipzig, 1989.
\newblock With German, French and Russian summaries.

\bibitem{Xu.J;Zikatanov.L1999}
J.~Xu and L.~Zikatanov.
\newblock A monotone finite element scheme for convection diffusion equations.
\newblock {\em Mathematics of Computation}, 68:1429--1446, 1999.

\bibitem{Young71}
D.~Young.
\newblock {\em Iterative solution of large linear systems}.
\newblock Academic Press, Inc., New York, 1971.

\end{thebibliography}

%\clearpage
%\input{app}

\vspace*{0.5cm}

\end{document}